\documentclass[a4paper,10pt]{amsart}
\usepackage{a4wide,amsmath,bbm,mathtools,mathrsfs,amsfonts,fancyhdr,color, ifthen, graphicx,stmaryrd, longtable,amssymb,latexsym,amsxtra,amscd,epic}

\DeclareMathOperator{\Aut}{Aut}
\DeclareMathOperator{\End}{End}

\DeclareMathOperator{\GL}{GL}
\DeclareMathOperator{\SL}{SL}
\DeclareMathOperator{\Sp}{Sp}
\DeclareMathOperator{\Hom}{Hom}

\DeclareMathOperator{\id}{id}
\DeclareMathOperator{\Wr}{Wr}

\DeclareMathOperator{\Ind}{Ind}

\DeclareMathOperator{\SO}{SO}

\DeclareMathOperator{\Sym}{Sym}
\DeclareMathOperator{\PV}{PV}
\DeclareMathOperator{\Res}{res}
\DeclareMathOperator{\Sol}{Sol}
\DeclareMathOperator{\rk}{rk}

\newtheorem{defin}{Definition}[section]
\newtheorem{thm}[defin]{Theorem}
\newtheorem{lemma}[defin]{Lemma}
\newtheorem{cor}[defin]{Corollary}
\newtheorem{prop}[defin]{Proposition}

\begin{document}
\thispagestyle{empty}
\flushleft
\title{Algebraic characterization of differential operators of Calabi-Yau type}
\author{Michael Bogner}
\address{Michael Bogner, Institut f\"ur Mathematik, Johannes Gutenberg-Universit\"at Mainz, Staudingerweg 9, 55128 Mainz, Germany.}
\address{Current address: 
Institut f\"ur Algebraische Geometrie, Leibniz Universit\"at
  Hannover, Welfengarten 1, 30167 Hannover, Germany}
\email{bognerm@uni-mainz.de}
\maketitle
\begin{abstract}
We give an algebraic characterization of Picard-Fuchs operators attached to families of Calabi-Yau manifolds with a point of maximally unipotent monodromy and discuss possibilities for their differential Galois groups. 
\end{abstract}

\section{Introduction}

One of the origins of the study of differential equations related to geometric objects is L. Euler's work on the hypergeometric differential equation
\[z(1-z)\frac{d^2y}{dz^2}+(c-(a+b+1)z)\frac{dy}{dz}-aby=0.\] 
In the special case $a=b=1/2,\ c=1$, all of its solutions can be written as $\mathbb{C}$-linear combinations of the elliptic integrals
\[y_1(z)=\int_0^{z}\frac{dx}{\sqrt{x(x-1)(x-z)}} \textrm{ and } y_2(z)=\int_{1}^{z}\frac{dx}{\sqrt{x(x-1)(x-z)}}.\]
Riemann's ideas led to the interpretation of these solutions as \textit{period integrals} of the family $y^2=x(x-1)(x-z)$ of elliptic curves, i.e. as variation of integrals of a holomorphic 1-form over a locally constant 1-cycle. 
More generally, if we consider arbitrary families of elliptic curves over a one-dimensional parameter space, it was observed by R. Fricke and F. Klein in \cite{Klein} that an appropriate associated family of periods $y(t)$ fulfills the hypergeometric differential equation
\[\frac{d^2y(t)}{dJ(t)^2}+\frac{1}{J(t)}\frac{dy(t)}{dJ(t)}+\frac{(31/144)J(t)-1/36}{J(t)^2(J(t)-1)^2}y(t)=0\] with respect to the $J$-function associated to the family. 

The concept of differential equations satisfied by period integrals - so called \textit{Picard-Fuchs equations} - was generalized to families of algebraic manifolds with higher dimensional fibers by E. Picard and H. Poincar\'e. P. Griffiths incorporated those studies into 20th century language of mathematics, see e.g. \cite{Griffiths} and \cite{TIT}, which leads to variations of Hodge structure and Gau\ss-Manin systems.  

As Riemann surfaces of genus one are Calabi-Yau manifolds, it is natural to study Picard-Fuchs equations for higher dimensional families of them. Contrary to the one-dimensional case, there is no universal way to describe these equations for higher dimensional families as Fricke and Klein did. 
For such families with two dimensional fibers, the order of a related Picard-Fuchs equation depends is given by $22$ minus the Picard rank of the generic fiber of the family, as the middle cohomology of the generic fiber has dimension $22$ and each integral of a holomorphic $(2,0)$-form over an algebraic cycle vanish. Those with Picard rank $19$ which yield a Picard-Fuchs operator of order three have e.g. been studied in \cite{Yui}. Calabi-Yau threefolds are of particular interest in mirror symmetry. In \cite{Can}, P. Candelas and his collaborators studied the family of quintics $X\subset \mathbb{P}^4$ and made an amazing discovery: 
By a classical result, such a manifold contains $2875$ lines and in fact, a parameter count suggests that $X$ contains for each degree $d$ a finite number $n_d$ of rational curves of degree $d$.
Candelas and his collaborators were able to predict these numbers starting with the Picard-Fuchs equation 
\[\left(z\frac{d}{dz}\right)^4-5z\left(5z\frac{d}{dz}+1\right)\left(5z\frac{d}{dz}+2\right)\left(5z\frac{d}{dz}+3\right)\left(5z\frac{d}{dz}+4\right)\] of the so called mirror manifold. 

Motivated by these studies, G. Almkvist, C. van Enckevort, D. van Straten and W. Zudilin started to collect differential operators which have similar algebraic properties and are potential Picard-Fuchs equations of families of Calabi-Yau threefolds which admit a large structure limit. These operators are called \textit{of CY-type}. The resulting first version of their list \cite[Appendix A]{AESZ} contained $306$ of such operators, but is updated constantly and in fact still growing. 

In this article, we establish a first algebraic characterization of Picard-Fuchs operators attached to special families of $n$-dimensional Calabi-Yau manifolds for each $n\in\mathbb{N}$, which treats all of the cases mentioned so far.    
By special, we mean that a section of $(n,0)$-forms gives rise to a sub variation of Hodge structure over $\mathbb{Q}$ which admits a limiting mixed Hodge structure of Hodge-Tate type. In fact, this gives rise to a differential operator of order $n+1$. The geometric situation together with some basic results in differential algebra are stated in the second section. In the third section, we discuss differential algebraic properties of related Picard-Fuchs operators. In particular, the Poincar\'e pairing on the smooth fibers of the family yield the self duality of a related Picard-Fuchs operator, i.e. relations on its coefficients. Moreover, the presence of a large structure limit leads the monodromy $T$ attached to one singular fiber to be maximally unipotent in the sense that $T-\id$ is nilpotent of maximal rank. The related point in the base is called a \textit{MUM point}. This has several impacts on the Picard-Fuchs operator, such as the equality of its exponents at this point and the existence of a local holomorphic solution whose Taylor series expansion has integral coefficients up to re-scaling. We also derive a local normal form at the MUM-point which enables us to read off potential instanton numbers related to possibly related families. Therefore, we claim further arithmetic properties on the local solutions of the differential operator. These properties yield to the purely differential algebraic notion of differential operators of CY-type. 
From an algebraic point of view, it is natural to study differential Galois groups of differential operators which is done in the fourth section. For a CY-type operator of order $n+1\geq 2$, its differential Galois group $G$ is a subgroup $\SL_{n+1}(\mathbb{C})$ if $n+1$ is even and of $\SO_{n+1}(\mathbb{C})$ if $n+1$ is odd. If $G$ is not the full symplectic or the full orthogonal group, a result of J.Saxl and G. Seitz, see \cite[Proposition 2.2]{SaxlSeitz}, implies that $G$ is either isomorphic to $\SL_2(\mathbb{C})$ or to the exceptional group $G_2(\mathbb{C})$, where the latter case is only possible if $n+1=7$. In section four, we discuss these possibilities and their impacts on the local normal forms. While all known examples with $G\cong \SL_2(\mathbb{C})$ can be achieved as symmetric powers of CY-type operators of order two, the first ones with $G\cong G_2(\mathbb{C})$ were found quite recently and have very interesting arithmetic properties.

\vspace{2ex}
\textbf{Acknowledgments:} I thank Duco van Straten, who introduced me to this field of research and under whose supervision I started to treat the topics discussed in this article as part of my PhD-thesis. I'm grateful to Stefan Reiter and J\"org Hofmann for several valuable discussions concerning CY-type differential equations. Moreover, I am indebted to Gert Almkvist for many years of constant support, interest and inspiration. A major part of the content of this article was established during my stay at the IMPA Rio de Janeiro, Brazil in August $2012$. I thank Hossein Movasati and the staff of the IMPA for offering their hospitality and excellent working conditions.

\section{Some basics}
\subsection{Geometric setting}\label{GS} 
Let us briefly recall the geometric situation we are interested in and fix some notation.

We consider a complex algebraic $(n+1)$-dimensional manifold $Y$ together with a proper morphism $\pi\colon Y\to\mathbb{P}^1$ with finite singular locus $S\subset \mathbb{P}^1$ whose regular fibers consist of smooth $n$-dimensional Calabi-Yau varieties. The \textit{Gau\ss-Manin connection} related to the $n$-th cohomology of the fibers is denoted by \[(\mathcal{H},\nabla):=(R^n\pi_{*}\underline{\mathbb{C}}_Y\otimes \mathcal{O}_{\mathbb{P}^1\setminus S},1\otimes d).\] Together with the underlying rational structure $R^n\pi_{*}\underline{\mathbb{Q}}_Y$ and the descending filtration of subsheaves $\mathcal{F}^{\bullet}$ of $\mathcal{H}$ which is induced by the Hodge filtrations on the fibers, this gives rise to a $\mathbb{Q}$-variation of Hodge structure. The primitive part of this variation of Hodge structure is polarized by the non-degenerate, $(-1)^n$-symmetric pairing \[(\cdot,\cdot)\colon \mathcal{H}\times\mathcal{H}\to \mathcal{O}_{\mathbb{P}^1\setminus S}\] induced by the Poincar\'e-pairing on the $n$-th cohomology of the fibers. 

By the choice of a coordinate $z$ on $\mathbb{P}^1\setminus S$, each local section $\omega\in \mathcal{H}(U)$ over an open set $U\subset \mathbb{P}^1\setminus S$ determines an annihilating differential operator of minimal degree with respect to $\nabla_{\frac{d}{dz}}$, whose coefficients lie in $\mathcal{O}_{\mathbb{P}^1\setminus S}(U)$. This operator is called the \textit{Picard-Fuchs operator} attached to $\omega$. If $\omega$ is a global section of $\mathcal{H}$, the Regularity Theorem \cite[Th\'eor\`eme II.7.9]{Del} by Griffiths and Deligne together with Serre's GAGA principle established in \cite{Serre} assure that the related Picard-Fuchs operator may be identified with a differential operator \[L:=\partial^{r+1}+\sum_{i=0}^ra_i\left(\frac{d}{dz}\right)^i\in\mathbb{Q}(z)\left[\frac{d}{dz}\right].\] 

\begin{defin}
By means of the procedure described above, a differential operator related to a global section $e$ of the rank one vector bundle $\mathcal{F}^n$ is called a \textbf{differential Calabi-Yau operator}.
\end{defin}
 
In the sequel, we only consider differential Calabi-Yau operators attached to families, for which each global section $e$ of $\mathcal{F}^n$ together with its derivatives $\nabla^i_{\frac{d}{dz}}(e)$ generates an irreducible subvariation of Hodge structure $(\mathcal{V},\nabla, \mathcal{F}^{\bullet}, \mathbb{V})$ where $\mathbb{V}$ is a local system of rank $n$. Consequently, its associated Picard-Fuchs operator $L$ is irreducible and of degree $n+1$. Furthermore, we assume that the monodromy $T\in \GL(\mathbb{V}_{x_0})$ around $z=0$ - with respect to an arbitrarily chosen base point $x_0\in\mathbb{P}^1\setminus S$ - is \textit{maximally unipotent}, i.e. $T-\id$ is nilpotent and $(T-\id)^{n}\neq 0$. In that case, we call $z=0$ a \textit{MUM point} of the family. 

\subsection{Differential setting}
To characterize differential Calabi-Yau operators from a purely algebraic point of view, let us recall some basic concepts of differential algebra which can be found in \cite[Chapter 1-2]{Put} and the references stated therein.

For $R\in\left\{\mathbb{C}[z],\mathbb{C}(z),\mathbb{C}\llbracket z\rrbracket, \mathbb{C}((z))\right\}$, we consider the differential ring $\left(R,\frac{d}{dz}\right)$ and denote its associated ring of differential operators by $R[\partial]$. Setting $\vartheta:=z\partial$, we also consider the differential ring $\mathbb{C}[z,\vartheta]:=\mathbb{C}[z][\vartheta]\subset \mathbb{C}(z)[\vartheta]=\mathbb{C}(z)[\partial]$. 
A \textit{differential module} over $R$ is an $R$-module $M$ of finite rank together with an additive map $\partial_M\colon M\to M$ which fulfills
\[\partial_M(fm)=f\partial_M(m)+\frac{df}{dz}m.\] To avoid inidices, we will usually abuse notation and write just $\partial$ instead of $\partial_{M}$. Differential modules form an abelian category, whose morphisms, so called \textit{differential morphisms}, are given by homomorphisms $\varphi\colon M\to N$ of $R$-modules which satisfy
\[\varphi(\partial(m))=\partial\left(\varphi(m)\right)\] for all $m\in M$. In particular, given two differential modules $M$ and $N$, their tensor product $M\otimes N$ is given by the tensor product of $R$-modules together with the map induced by \[\partial(m\otimes n)=\partial(m)\otimes n+m\otimes \partial(n)\] and the homomorphisms $\Hom(M,N)$ are given by the homomorphisms of $R$-modules together with the map induced by 
\[\partial(\psi)(m)=\partial(\psi(m))-\psi(\partial(m)).\] In this sense, the differential morphisms of $M$ and $N$ are precisely those $\varphi\in\Hom(M,N)$ with $\partial(\varphi)=0$.

If $M$ is a differential module and $e\in M$ such that $\{\partial^ie\}_{i\in\mathbb{N}}$ generates $M$ as an $R$-module, we call $e$ a \textit{cyclic vector} of $M$ and the pair $(M,e)$ a \textit{marked differential module}. In this case, there is a unique reduced differential operator $L\in R[\partial]$ such that the differential morphism induced by
\[M\to R[\partial]/R[\partial]L,\ e\mapsto 1\]
is an isomorphism of differential modules.
This operator $L$ is called the related \textit{minimal operator} of $e$. Conversely, given a reduced differential operator $L\in R[\partial]$, we call $\left(R[\partial]/R[\partial]L,1\right)$ its associated marked differential module.

\begin{defin}
The marked differential module associated to a differential Calabi-Yau operator is called a \textbf{differential Calabi-Yau module}. 
\end{defin}

We also recall some facts concerning solutions, regularity and differential Galois groups of differential $R$-modules. Given an extension $S\supset R$ of simple differential rings and a differential $R$-module $M$, its \textit{solution space} with values in $S$ is given by the $\mathbb{C}$-vectorspace
\[\Sol_S(M):=\ker(\partial, M\otimes_R S).\] By a classical wronskian argument, we have \[\dim_{\mathbb{C}}(\Sol_S(M))\leq \rk_R(M).\]
There is an up to isomorphism uniquely determined minimal simple differential ring $\mathcal{F}$ such that $\dim_{\mathbb{C}}(\Sol_{\mathcal{F}}(M))=\rk_R(M)$  for each differential $R$-module $M$, the so called \textit{universal Picard-Vessiot ring} of $R$. 
For a specific differential $k$-module $M$, we call the minimal differential subring $\PV(M)\subset\mathcal{F}$ such that 
$\dim_{\mathbb{C}}\left(\Sol_{\PV(M)}(M)\right)=\rk_{R} (M)$ the \textit{Picard-Vessiot ring} of $M$.
In the sequel, we write \[\Sol(M):=\Sol_{\mathcal{F}}(M)=\Sol_{\PV(M)}(M).\] We have a natural differential isomorphism
\[\Sol(M)\otimes_{\mathbb{C}}\PV(M)\cong M\otimes_{k} \PV(M).\] 
Moreover, the $\mathbb{C}$-linear map
\[\Sol(M)^{\vee}\to\Sol(L),\ \psi\mapsto\psi(e)\] is an isomorphism of $\mathbb{C}$-vectorspaces.
For each marked differential $R$-module $(M,e)$ with associated minimal operator $L$, we call $\Sol(L):=\{y\in \PV(M)\mid L(y)=0\}$ the solution space of $L$.

We call a differential $\mathbb{C}((z))$-module $M$ to be \textit{regular singular}, if its Picard-Vessiot ring is contained in 
\[\mathcal{F}^{rs}:=\mathbb{C}\{z\}\left[\{z^a\}_{a\in\mathbb{C}}, \ln(z)\right],\] where $\mathbb{C}\{z\}$ denotes the ring of convergent power series. 
A differential $\mathbb{C}(z)$-module $M$ is \textit{fuchsian}, if its natural localization with respect to each $s\in\mathbb{P}^1$ is regular singular. Similarly, a monic differential operator $L=\partial^{n+1}+\sum_{i=0}^na_i\partial^i\in\mathbb{C}(z)$ is \textit{fuchsian} if its associated marked differential module is. Its \textit{indicial equation} at $z=p$ is given by
\[\Ind_p(L)(T):=\prod_{j=0}^n(T-j)+\sum_{i=0}^n\Res_{z=p}\left((z-p)^{n-i}a_i\right)\prod_{j=0}^i(T-j)\in\mathbb{C}[T].\]
The roots of $\Ind_p(L)$ are called the \textit{exponents} of $L$ at $z=p$. In fact, if $\mu$ is an exponent of $L$ at $z=p\in\mathbb{C}$ then $L$ admits a local solution $f\in (z-p)^{\mu}\mathbb{C}\{z-p\}$. If there is no exponent $\nu$ such that $\nu-\mu\in\mathbb{N}$, we even have $f\in (z-p)^{\mu}\mathbb{C}\{z-p\}^{*}$. 

By multiplication with a polynomial $g\in\mathbb{C}[z]$ from the left, we may write
\[gL=\sum_{i=0}^mz^iP_i\in\mathbb{C}[z][\vartheta],\ \textrm{ where }P_0,\dots,P_m\in\mathbb{C}[\vartheta].\]
Then $P_0=\Ind_0(L)(T)$ and $P_n=\Ind_{\infty}(L)(-T)$.  
In the sequel, we write $\mathbb{C}[z][\vartheta]=:\mathbb{C}[z,\vartheta]$ and do not distinguish between $L\in\mathbb{C}(z)[\partial]$ and the corresponding $gL\in\mathbb{C}[z][\vartheta]$ as above, since both operators have the same solutions. 
If $L$ is monic and $f\in\mathcal{F}$ is a solution of $L$, we have $f\in\mathbb{C}\{z-x_0\}$ if $x_0\in\mathbb{C}$ is not amongst the singularities $S\subset\mathbb{P}^1$ of the coefficients of $L$.
Choosing such a point $x_0$, the \textit{local monodromy} of $L$ at $p\in\mathbb{P}^1$ is given by the effect of analytic continuation of the local solutions of $L$ near $x_0$ around any closed path which starts at $x_0$, encircles $p$ in counterclockwise direction and encircles no point of $S\setminus \{p\}$.

Each differential $R$-module $M$ has a \textit{differential Galois group} $G(M):=\Aut_{R[\partial]}(\PV(M))$. Via its natural action on $\Sol(M)$, this group is isomorphic to an algebraic subgroup of $\GL_{\rk_R(M)}(\mathbb{C})$. If $L$ is fuchsian, the Zariski closure of the group generated by its local monodromies is the differential Galois group of its associated marked differential module, see \cite[Theorem 5.8]{Put}. The natural action of the \textit{absolute Galois group} $\mathcal{G}:=\Aut_{R[\partial]}(\mathcal{F})$ on the solution space of each differential $R$-module turns the category of differential $R$-modules into a neutral Tannaka category, see e.g. \cite[Appendix B]{Put}

\section{Algebraic characterization}
Let in the whole section $L:=\partial^{n+1}+\sum_{i=0}^{n}a_i\partial^i\in\mathbb{Q}(z)[\partial]$ be irreducible of degree $n+1$ and $(M_L,e)$ be the associated marked differential module.
Moreover, we stick to the notation of Section \ref{GS}.
\subsection{The Poincar\'e pairing}

We first discuss the effect of the Poincar\'e pairing $(\cdot,\cdot)$ if $(M_L,e)$ is a differential Calabi-Yau module. As the pairing is compatible with the connection in the sense that we have \[\frac{d}{dz}\left(h_1, h_2\right)=\left(\nabla_{\frac{d}{dz}}(h_1),h_2\right)+\left(h_1, \nabla_{\frac{d}{dz}}(h_2)\right)\] for all local sections $h_1, h_2$ of $\mathcal{H}$, it gives rise to a non-degenerate pairing $\langle\cdot,\cdot\rangle\in \Sol_{\mathbb{C}(z)}\left(\bigwedge^2M_L\right)^{\vee}$
if $n$ is odd and $\langle\cdot,\cdot\rangle\in \Sol_{\mathbb{C}(z)}\left(\Sym^2M_L\right)^{\vee}$ if $n$ is even. Moreover, as $e$ is a section of $(n,0)$-forms, the Griffiths transversality-property \[\nabla\left(\mathcal{F}^{\bullet}\right)\subset \mathcal{F}^{\bullet\,-1}\otimes \Omega^1_{\mathbb{P}^1\setminus S}\] assures that $\langle e,\partial^ie\rangle=0$ holds for $0\leq i\leq n-1$.

\begin{defin}
We say that $L$ satisfies \textbf{property (P)} if there is a non-degenerate form \[\langle\cdot,\cdot\rangle\colon M_L\times M_L\to \mathbb{C}(z)\] such that
\begin{enumerate}
 \item $\langle\cdot,\cdot\rangle\in\Sol_{\mathbb{C}(z)}\left(\Sym^2M_L\right)^{\vee}$ for $n$ even and \mbox{$\langle\cdot,\cdot\rangle\in \Sol_{\mathbb{C}(z)}\left(\bigwedge^2M_L\right)^{\vee}$} for $n$ odd.
\item $\langle e,\partial^ie\rangle=0$ for $i=0,\dots,n-1$. 
\end{enumerate}
\end{defin}

This condition induces relations on the coefficients of $L$. To express them clearly, we introduce the notion of the dual of a differential operator.

\begin{defin}
The \textbf{dual} of $L$ is given by
\[L^{\vee}:=\sum_{i=0}^d (-1)^{n+1+i}\partial^ia_i\in\mathbb{Q}(z)[\partial].\] 
\end{defin}

The dual of a differential operator is related to the dual $M_L^{\vee}=\Hom(M_L,\mathbb{C}(z))$ of its induced marked differential module in the following way, see \cite[Section 2.2]{Put}:

\begin{lemma}\label{Dualpair}
Let $B=\{e,\dots,\partial^{n}e\}$ and $B^{\vee}=\left\{e^{\vee},\dots,\left(\partial^{n}e\right)^{\vee}\right\}$ the basis of the dual differential module $M_L^{\vee}$ which is dual to $B$. Then $\left(M_L^{\vee}, \left(\partial^{n}e\right)^{\vee}\right)$ is a marked differential module with minimal operator $L^{\vee}$.
\end{lemma}

The lemma above enables us to express property (P) in the context of differential operators.

\begin{prop}\label{CY2anders}
The operator $L$ satisfies property (P) if and only if $L$ is self-dual in the sense that there is a $0\not=\alpha\in\mathbb{C}(z)$ such that
$L\alpha=\alpha L^{\vee}$ holds. In particular, we have $\alpha'=-2a_n\alpha/(n+1).$  
\end{prop}

\begin{proof}
Suppose that $L$ satisfies property (P). As $\langle\cdot,\cdot\rangle$ is non-degenerate, we have \[\alpha:=\langle e,\partial^{n}e\rangle\not=0.\] The map 
\[\varphi\colon M_L\to M_L^{\vee},\ \varphi(m):=\langle m,\cdot\rangle\] is an isomorphism of differential modules. With respect to the basis $\{e^{\vee},\dots,\left(\partial^{n}e\right)^{\vee}\}$ dual to $\{e,\dots,\partial^ne\}$, the element $\varphi(e)$ is exactly $\alpha\left(\partial^{n}e\right)^{\vee}$. By \cite[Lemma 2.5]{Singer} and Lemma \ref{Dualpair}, this implies $L\alpha=\alpha L^{\vee}$. 

Conversely, if $L\alpha=\alpha L^{\vee}$ for $\alpha\not=0$, we have a corresponding differential isomorphism \[\psi\colon M_L\to M_L^{\vee},\ \psi(e):=\alpha\left(\partial^{n}e\right)^{\vee}.\] Thus the form \[\langle\cdot,\cdot\rangle\colon M_L\times M_L\to \mathbb{C}(z),\ \langle m_1, m_2\rangle:=\psi(m_1)(m_2)\] is non-degenerate and satisfies $\langle e,\partial^{i} e\rangle=0$ for all $i=0,\dots,n-1$. As $\psi$ is a differential morphism, we have \[\frac{d}{dz}\langle m_1,m_2\rangle=\langle\partial m_1,m_2\rangle+\langle m_1,\partial m_2\rangle\] for each two elements $m_1,m_2\in M_L$.
It remains to prove that $\langle\cdot,\cdot\rangle$ is $(-1)^{n}$-symmetric. 
By construction of $\langle\cdot,\cdot \rangle$, we have $\langle e,\partial^ie\rangle=0$ for $0\leq i<n$ and $\langle e,\partial^ne\rangle=\alpha$. By definition, we have $\partial^i\left(\langle e,\cdot\rangle\right)=\langle\partial^i e,\cdot\rangle$. 
Computing the derivatives of $\langle e,\cdot\rangle$ explicitly, we conclude that
$\langle\partial^ie,e\rangle=0$ for $0\leq i<n$ and $\langle\partial^ne,e\rangle=(-1)^{n}\alpha$. 
Fix $N\leq n$ and suppose that we have shown the identity \[\langle \partial^me,\partial^ke\rangle=(-1)^{n}\langle \partial^ke,\partial^m e\rangle\] for every $m\leq N$ and every $0\leq k\leq n$. We get
\begin{align*}
\langle\partial^{N+1}e,\partial^ke\rangle&=\frac{d}{dz}\langle\partial^Ne,\partial^ke\rangle-\langle\partial^Ne,\partial^{k+1}e\rangle\\&=(-1)^{n}\frac{d}{dz}\langle\partial^ke,\partial^Ne\rangle+(-1)^{n+1}\langle\partial^{k+1}e,\partial^Ne\rangle=(-1)^{n}\langle\partial^k e,\partial^{N+1}e\rangle.
\end{align*}
Thus the desired result follows by induction. 
Finally, comparing the coefficients of $L$ and $L^{\vee}$, one readily sees that $\alpha'=-2a_n\alpha/(n+1)$ holds.   
\end{proof}

We derive some corollaries and facts related to property (P).
First, we investigate the effect on the exponents of $L$, as soon as all of them are real and $L$ is fuchsian.

\begin{cor}\label{Exponentssymm}
For each fuchsian operator $L$ that satisfies property (P) and has real exponents $\lambda_{1}\leq\dots\leq\lambda_{n+1}$ at $s\in \mathbb{P}^1$, we have
\[\frac{2}{n+1}\sum_{i=1}^{n+1}\lambda_{i}\in\mathbb{Z}\] and
\[\lambda_{i}+\lambda_{n+2-i}=\lambda_{j}+\lambda_{n+2-j}\] for all $1\leq i,j\leq n+1$.
\end{cor}

\begin{proof}
As $L=\partial^{n+1}+\sum_{i=0}^{n}a_{i}\partial^i$ satisfies property (P), we have $\alpha^{-1}L\alpha=L^{\vee}$ for a  $0\neq\alpha\in\mathbb{C}(z)$ with $\alpha'=-2a_{n}\alpha/(n+1)$. Solving this equation locally reveals that $2/(n+1)\Res_{z=s}(a_{n})\in\mathbb{Z}.$ 
As $\lambda_1,\dots,\lambda_{n+1}$ are precisely the roots of
\[\Ind_s(L)(T):=\prod_{j=0}^n(T-j)+\sum_{i=0}^n\Res_{z=s}\left((z-s)^{n-i}a_i\right)\prod_{j=0}^i(T-j),\] we find that
\[\binom{n+1}{2}-\sum_{i=1}^{n+1}\lambda_{i}=\Res_{z=s}(a_{n})\] which implies the first result.
For the second statement, assume without loss of generality that $s=0$ and write $gL=\sum_{i=0}^{m}z^iP_i(\vartheta)\in\mathbb{C}[z,\vartheta]$. By the rules $(PQ)^{\vee}=Q^{\vee}P^{\vee}$, $\vartheta^{\vee}=-\vartheta-1$ and $\vartheta z^i=z^i(\vartheta+i)$, we get that \[(gL)^{\vee}=L^{\vee}g=\sum_{i=0}^mz^iP_i(-1-\vartheta-i).\] In particular, there is an $a\in\mathbb{Q}$ such that \[\lambda_{i}-a=-\lambda_{n+2-i}\] holds for each $1\leq i\leq n+1$. As this implies $\lambda_{i}+\lambda_{n+2-i}=a$, we get the second result.
\end{proof}

Next, we note that $\langle\cdot,\cdot\rangle$ induces a non-degenerate, $(-1)^n$-symmetric form
\[\langle\cdot,\cdot\rangle_{\PV}\colon\left( M_L\otimes \PV(M_L)\right)\times \left(M_L\otimes \PV(M_L)\right)\to \PV(M_L).\] As the solution space of $M_L$ is a $\mathbb{C}$-lattice in $M_L\otimes \PV(M_L)$, this gives rise to a non-degenerate, $(-1)^n$-symmetric form 
\[\langle\cdot,\cdot\rangle_{\mathbb{C}}\colon \Sol(M_L)\times \Sol(M_L)\to \mathbb{C}.\]  As a direct consequence concerning the differential Galois group, we get:

\begin{cor}\label{Gfirst}
If $L$ satisfies property (P), the differential Galois group of $M_L$ lies in $\Sp_{n+1}(\mathbb{C})$ for $n+1$ even and in $\SO_{n+1}(\mathbb{C})$ for $n+1$ odd. 
\end{cor}

As the category of differential $\mathbb{C}(z)$-modules is tannakian, we get submodules of $\bigwedge^2 M_L$ and $\Sym^2 M_L$.

\begin{cor}\label{ReprCY}
Suppose that $L$ satisfies property (P).
\begin{enumerate}
 \item If $\deg(L)>2$ is even, $\bigwedge^2 M_L$ has an one dimensional differential submodule $W$ which is not contained in the differential submodule generated by $e\wedge\partial e$.
\item If $\deg(L)>1$ is odd, $\Sym^2 M_L$ has an one dimensional differential submodule $W$ which is not contained in the differential submodule generated by $e\cdot e$.
\end{enumerate}
\end{cor}

\begin{proof}
We only state the proof for $\deg(L)>2$ even, as the odd case can be treated similarly. The $\mathbb{C}$-span of $\langle\cdot,\cdot\rangle$ is a $G(M_L)$-invariant subspace of $\Sol\left(\bigwedge^2M_L\right)$. By \cite[Corollary 2.35]{Put} this subspace gives rise to a one dimensional differential submodule $W$ of $\bigwedge^2M_L$, as $\dim_{\mathbb{C}(z)}\left(\bigwedge^2M_L\right)>1$ . In particular, $\Sol_{\mathbb{C}(z)}(W)^{\vee}$ is spanned by $\langle\cdot,\cdot\rangle$. Let $N$ be the differential submodule of $\bigwedge^2M_L$ which is generated by $e\wedge\partial e$ and assume that $W\subset N$. As then the natural map \[\Sol_{\mathbb{C}(z)}(N)^{\vee}\to\Sol_{\mathbb{C}(z)}(W)^{\vee}\] is surjective and the restriction of $\langle\cdot,\cdot\rangle$ to $N$ is identically zero, this is impossible.
\end{proof}

\subsection{Exponents at the MUM point}

As we have claimed that the monodromy of the family $\pi\colon Y\to\mathbb{P}^1\setminus S$ at $z=0$ is maximally unipotent the same holds for the monodromy $T_0$ of $L$ at $z=0$. As each eigenvalue of $T_0$ is of the form $\exp(2\pi i \lambda)$ where $\lambda$ is an exponent of $L$ at $z=0$, all exponents of $L$ at $z=0$ are integers. 

In our situation, we are even able to proof

\begin{prop}\label{ExpMUM}
If $L$ is a differential Calabi-Yau operator, its exponents at $z=0$ are all equal.
\end{prop} 

\begin{proof}
We pass to the local situation over the punctured disc $\Delta^{*}$ centered at $z=0$. Denote by $(\mathcal{V},\nabla,\mathcal{F}^{\bullet})$ the variation of Hodge structure induced by the $\nabla_{\frac{d}{dz}}^i(e)$. The proof makes use of the limiting mixed Hodge structure introduced by W. Schmid in \cite{Schmid} and related facts taken from \cite[Section II]{Kul}. 
By the Regularity Theorem \cite[Th\'eor\`eme II.7.9]{Del}, the connection $(\mathcal{V},\nabla)$ admits a canonical extension to a regular singular connection $\left(\overline{\mathcal{V}},\overline{\nabla}\right)$ on $\Delta$. In particular, there is a frame $F\subset \overline{\mathcal{V}}(\Delta)$ with associated lattice $\Sigma=\mathbb{C}\{z\}F$ such that 
$\overline{\nabla}_{z\frac{d}{dz}}(\Sigma)\subset \Sigma$ holds and the real parts of the eigenvalues of the Euler operator \[E_{\Sigma}\in \End_{\mathbb{C}}\left(\Sigma/z\Sigma\right),\ [\sigma]\mapsto \left[\overline{\nabla}_{z\frac{d}{dz}}(\sigma)\right]\] lie in $(-1,0]$. 
Setting $V_{\infty}:=\Gamma(\mathfrak{h},e^{*}\mathbb{V})$, where  $\mathfrak{h}\subset \mathbb{C}$ is the upper half plane and $e$ the universal cover of $\Delta^{*}$, we get an isomorphism $\psi_z\colon V_{\infty}\to \Sigma/z\Sigma$ of $\mathbb{C}$-vectorspaces.
Denote the natural lift of $T_0$ on $V_{\infty}$ by $T_{\infty}$ and its logarithm by \[N=\log(T_{\infty})=\sum_{k=1}^{\infty}\frac{(-1)^{k-1}}{k}(T_{\infty}-\id)^k\in\GL(V_{\infty}).\]
Then, with respect to $\psi_z$, we have \[E_{\Sigma}=-\frac{1}{2\pi i}N.\] In particular, as $T_{\infty}$ is maximally unipotent, the characteristic polynomial of $E_{\Sigma}$ is $X^{n+1}$. Moreover, as described in \cite[Section 6]{Schmid}, the nilpotent map $N$ determines the \textit{monodromy weight filtration} $W_{\bullet}$ and the filtration $\mathcal{F}^{\bullet}$ admits an extension $F^{\bullet}_{\infty}$ to $V_{\infty}$ such that $(V_{\infty}, F^{\bullet}_{\infty}, W_{\bullet})$ is a mixed Hodge structure.

As $\mathcal{F}^n$ is a vector bundle of rank one, there is a section $g\in\mathcal{O}_{\Delta}$ such that the germ $e'(0)$ of $e':=g\overline{e}$ at $z=0$ lies in the frame $F$, where $\overline{e}\in\overline{\mathcal{V}}(\Delta)$ denotes the extension of $e$. In particular, there is an integer $k\in\mathbb{Z}$ such that the exponents of $L$ at $z=0$ are given by $\lambda_1+k,\dots, \lambda_{n+1}+k$ where $\lambda_1,\dots,\lambda_{n+1}$ are the exponents of the minimal operator of $e'$ at $z=0$. 
    
Denote the class of $v\in \Sigma$ in $\Sigma/z\Sigma$ by $[v]$. 
As $W_{2n}=\ker\left(N^{n}\right)$ and $F^{n}_{\infty}\oplus W_{2n}=V_{\infty}$, we have $N^{n}[e'(0)]\neq 0$.
Therefore, the elements $[e'(0)], E_{\Sigma}[e'(0)],\dots, E_{\Sigma}^{n}[e'(0)]$ form a basis of $V_{\infty}$. Considering the action of the Euler operator $E_{\Sigma}$ on $V_{\infty}$ with respect to this basis, we readily see that its characteristic polynomial equals the indicial equation of the minimal operator of $e'$ at $z=0$. This yields the result.
\end{proof}

Moreover, as a direct consequence of the classical method of Frobenius to compute local solutions, see e.g. \cite[Chapter 16]{Ince}, we have

\begin{lemma}
If $\Ind_0(L)=(T-r)^{n+1}$ with $r\in\mathbb{Z}$, the formal monodromy of $L$ at $z=0$ is maximally unipotent. 
\end{lemma}
 
It is hence reasonable to take the result of Proposition \ref{ExpMUM} into account for our algebraic characterization.

\begin{defin}
We say that $L$ satisfies \textbf{property (M)} if there is an $r\in\mathbb{Z}$ such that its indicial equation at $z=0$ reads $(T-r)^{n+1}\in\mathbb{C}[T]$. 
\end{defin}

We also consider special bases of the solution space of $L$ near $z=0$.

\begin{defin}
We call a basis $\{y_0,\dots,y_n\}$ of the solution space of $L$ near $z=0$ a \textbf{flag} if there are elements $f_0,\dots,f_n\in\mathbb{C}\llbracket z\rrbracket$ such that $z^{-r}f_0=1+z\mathbb{C}\llbracket z\rrbracket$ and \[y_k=\sum_{j=0}^k\frac{1}{j!}\ln(z)^jf_{k-j}\] holds. 
\end{defin}

\subsection{Solutions at the MUM point}

By the underlying $\mathbb{Q}$-structure of our family, we actually suppose the related differential Calabi-Yau operator to be an element of $\mathbb{Q}(z)[\partial]$. According to \cite[Appendix V]{Andre}, each of its local holomorphic solution $y=\sum_{m=0}^{\infty}A_m(z-p)^m\in\overline{\mathbb{Q}}\llbracket z-p\rrbracket$ at a point $p$ in the algebraic closure $\overline{\mathbb{Q}}$ of $\mathbb{Q}$ in $\mathbb{C}$ is a \textit{G-function}, i.e.
\begin{enumerate}
\item it has a positive radius of convergence in $\mathbb{C}$.
\item there is a differential operator $P\in\overline{\mathbb{Q}}[z,\vartheta]$, such that $P(y)=0$.
\item there is a sequence of positive integers $(c_n)_{n\in\mathbb{N}}$ and an algebraic number field $K\subset\overline{\mathbb{Q}}$ such that \[\sup_{n\in\mathbb{N}}\left(\frac{1}{n}\ln(c_n)\right)<\infty\] and $c_nA_j\in\mathcal{O}_K$ for all $j\leq n$.
\end{enumerate} 

A differential operator which has a local solution which is a G-function is called a \textit{G-operator}. By a deep theorem of the Chudnovskies, see \cite[Chapter VIII]{Dwork}, each local holomorphic solution of an irreducible G-operator at a point in $\overline{\mathbb{Q}}$ is a G-function. In combination with a theorem of N.Katz, see \cite[Chapter III.6]{Dwork}, this assures that irreducible G-operators are fuchsian and have only rational exponents, i.e. the local monodromies are quasi-unipotent. For the characterization we are after, it is hence reasonable to request that the operators we investigate are G-operators.

Y. Andr\'e's result \cite[Theorem IX.4.2]{Andre} implies that the holomorphic solutions at a MUM point are of a very special shape:

\begin{thm}
Each holomorphic solution $y=\sum_{m=0}^{\infty}A_mz^m\in\mathbb{Q}\llbracket z\rrbracket$ of a differential Calabi-Yau operator $L$ at $z=0$ is \textbf{globally bounded}, i.e. there is an integer $N\in\mathbb{N}$ such that $A_m\in\mathbb{Z}\left[\frac{1}{N}\right]$ for all $m\geq 0$.
\end{thm}

Globally bounded G-functions with coefficients in $\mathbb{Q}$ are of the following shape.

\begin{defin}
A formal power series $y=\sum_{m=0}^{\infty}A_mz^m\in\mathbb{Q}\llbracket z\rrbracket$ is called \textbf{N-integral}, if there is an $N\in\mathbb{N}$ such that $N^mA_m\in\mathbb{Z}$. 
\end{defin}

\begin{lemma}
Consider a G-function $y=\sum_{m=0}^{\infty}A_mz^m\in\mathbb{Q}\llbracket z\rrbracket$. Then $y$ is globally bounded if and only if it is N-integral. 
\end{lemma}

\begin{proof}
It is clear that N-integral G-functions are globally bounded. Conversely, if $y$ is globally bounded, there is a sequence of integers $(d_n)_{n\in\mathbb{N}}$ and an $N\in\mathbb{N}$ such that $N^{d_n}A_j\in\mathbb{Z}$ for each $0\leq j\leq n$ and $\sup_{n\in\mathbb{N}}\left(\frac{1}{n}d_n\ln(N)\right)<\infty$. Hence, there is a $C\in\mathbb{N}$ such that $d_n\leq Cn$ for all $n\in\mathbb{N}$ which implies the result.
\end{proof}

We take this property into account for our description.

\begin{defin}
We say that $L$ satisfies \textbf{property (N)} if it has an N-integral solution at $z=0$.
\end{defin} 

Note, that each operator $L$ which satisfies (N) and (M) is a G-operator.  
Important prototypes of N-integral power series are Taylor series of algebraic functions at a point where the function is holomorphic, see e.g. \cite{Eisen}, and hypergeometric functions 
$_nF_{n-1}(\alpha_1,\dots,\alpha_n; 1,\dots,1\mid z)$. Moreover, the class of N-integral power series is closed under formal derivation, inversion, composition if possible, Cauchy products and Hadamard products.

\subsection{The local normal form at the MUM point}

For those $L$ which satisfy properties (M) and (P), we derive a local normal form at the MUM point $z=0$ and discuss some of its applications. In particular, we construct a special differential operator in $\mathbb{Q}\llbracket z\rrbracket[\vartheta]$ which has the same solutions as $L$ near $z=0$ and hence has to coincide with $L$ up to left multiplication by a formal power series. We suppose without loss of generality that zero is the only exponent of $L$ at $z=0$ and fix a flag $y_k=\sum_{j=0}^k\frac{1}{j!}\ln(z)^jf_{k-j}$, where $0\leq k\leq n$, of $L$ at $z=0$. 

As $y_0=f_0\in\mathbb{Q}\llbracket z\rrbracket^{*}$, we have 
\[\mathcal{N}_1:=\vartheta\frac{1}{y_0}\in\mathbb{Q}\llbracket z\rrbracket[\vartheta]\] and $\mathcal{N}_1(y_0)=0$.
Moreover, we see that 
\[\mathcal{N}_1(y_1)=z\frac{d}{dz}\left(\frac{\ln(z)f_0+f_1}{f_0}\right)=1+z\frac{d}{dz}\left(\frac{f_1}{f_0}\right)\in\mathbb{Q}\llbracket z\rrbracket^{*}.\]
In particular, we get
\[\mathcal{N}_2:=\vartheta\frac{1}{\mathcal{N}_1(y_1)}\vartheta\frac{1}{y_0}\in\mathbb{Q}\llbracket z\rrbracket[\vartheta]\] and $\mathcal{N}_2(y_0)=\mathcal{N}_2(y_1)=0$.

An iteration of this process yields

\begin{lemma}\label{NorFor}
With respect to the notation above, put \[\mathcal{N}_0:=1\] and
\[\mathcal{N}_{k+1}:=\vartheta\frac{1}{\mathcal{N}_{k}(y_{k})}\mathcal{N}_k.\]
Then we have $\mathcal{N}_j\in\mathbb{Q}\llbracket z\rrbracket[\vartheta]$ for all $0\leq j\leq n+1$ and $\mathcal{N}_j(y_i)=0$ for all $0\leq i\leq j-1$.
\end{lemma}  

\begin{proof}
We check the result inductively. Suppose that $\mathcal{N}_k\in\mathbb{Q}\llbracket z\rrbracket[\vartheta]$, $\mathcal{N}_{k}(y_i)=0$ for all $0\leq i\leq k-1$ and that there are
$f^{(k)}_{k},\dots,f^{(k)}_n\in\mathbb{Q}\llbracket z\rrbracket$, $f_{k}^{(k)}\in\mathbb{Q}\llbracket z\rrbracket^{*}$, such that
\[\mathcal{N}_k\left(y_{k+j}\right)=\sum_{i=0}^{j}\frac{1}{i!}\ln(z)^if^{(k)}_{k+j-i}.\] In particular, we have $\mathcal{N}_{k}(y_{k})=f^{(k)}_{k}\in\mathbb{Q}\llbracket z\rrbracket^{*}$. This implies $\mathcal{N}_{k+1}\in\mathbb{Q}\llbracket z\rrbracket[\vartheta]$ and $\mathcal{N}_{k+1}(y_i)=0$ for all $0\leq i\leq k$. Moreover, we get
\begin{align*}
\mathcal{N}_{k+1}\left(y_{k+1+j}\right)&=z\frac{d}{dz}\left(\sum_{i=0}^{j+1}\frac{1}{i!}\ln(z)^i\frac{f^{(k)}_{k+1+j-i}}{f^{(k)}_{k}}\right)\\&=\sum_{i=0}^{j}\frac{1}{i!}\ln(z)^i\left(\frac{f^{(k)}_{k+j-i}}{f^{(k)}_{k}}+z\frac{d}{dz}\left(\frac{f^{(k)}_{k+1+j-i}}{f^{(k)}_{k}}\right)\right)
\\&:=\sum_{i=0}^{j}\frac{1}{i!}\ln(z)f^{(k+1)}_{k+1+j-i}.
\end{align*} 
As additionally \[f^{(k+1)}_{k+1}=1+z\frac{d}{dz}\left(\frac{f^{(k)}_{k+1}}{f^{(k)}_{k}}\right)\in\mathbb{Q}\llbracket z\rrbracket^{*},\] we get the desired result.
\end{proof}

One can check directly that the differential operators $\mathcal{N}_k$ appearing in the procedure above do not depend on the initial choice of the flag. As by construction the local solutions of $\mathcal{N}_{n+1}$ at $z=0$ coincide with those of $L$, we say

\begin{defin}
The operator \[\mathcal{N}(L):=\mathcal{N}_{n+1}y_0=\vartheta\frac{1}{\mathcal{N}_{n}(y_{n})}\vartheta\frac{1}{\mathcal{N}_{n-1}(y_{n-1})}\dots\vartheta\frac{1}{\mathcal{N}_1(y_1)}\vartheta\in\mathbb{Q}\llbracket z\rrbracket[\vartheta]\] constructed by the procedure described in Lemma \ref{NorFor} is called the \textbf{local normal form} of $L$ at $z=0$. Moreover, we call \[\alpha_i:=\mathcal{N}_{i}(y_i)^{-1}\in\mathbb{Q}\llbracket z\rrbracket^{*}\] the $i$-th \textbf {structure series} of $L$ for $1\leq i\leq n$.
\end{defin}

By \cite[Section 6]{Del2} the local normal form can be interpreted by means of Hodge theory. In the geometric situation, given the limiting mixed Hodge structure $(V_{\infty}, F^{\bullet}_{\infty}, W_{\bullet})$, the subspaces $V^p=F^{p}_{\infty}\cap W_{2p}$ induce a splitting $V_{\infty}=\oplus_{p\geq 0}V^p$.

We denote the Picard-Vessiot ring of the differential $\mathbb{C}\llbracket z\rrbracket$-module $\kappa_0(M_L):=M_L\otimes_{\mathbb{C}(z)}\mathbb{C}\llbracket z\rrbracket$ by $\PV_0$ and denote the canonical extension of $\langle\cdot,\cdot\rangle$ to $M_L\otimes \PV_0$ by the same symbol. 
In our situation, we consider natural descending filtration $E^{\bullet}$ on $M_L\otimes \PV_0$ given by
\[E^{k}:=span\{e,\dots,\partial^{n-k}e\}.\] 
As a direct consequence of property (P), we have \[\left(E^{i}\right)^{\bot}=E^{n+1-i}\] with respect to $\langle\cdot,\cdot\rangle_{\PV_0}$ for every $i=0,\dots,n-1$.
We also consider the elements $h_0,\dots,h_n\in \Sol(\kappa_0(M_L))$ which are determined by \[\langle h_i,e\rangle=y_i.\] These give rise to an 
ascending filtration \[W_{2k}=W_{2k+1}:=span\{h_0,\dots,h_k\}\] on $M_L\otimes \PV_0$ which does not depend on the previous choice of a flag.
Moreover, we have \[W_{i}^{\perp}=W_{2n-1-i} \] with respect to $\langle\cdot,\cdot\rangle_{\PV_0}$ for each $i\geq 0$.
We put
\[V^p:=E^p\cap W_{2p}\] for $0\leq p\leq n$ and show that this gives a splitting of $\kappa_0(M_L)$.

\begin{prop}\label{Vau}
With respect to the notation of Lemma \ref{NorFor} and the filtrations introduced above, we have \[\mathcal{N}_k(e)\in V^{n-k}.\]
\end{prop}

\begin{proof}
It is clear that $\mathcal{N}_k(e)\in E^{n-k}$. Write $e=\sum_{i=0}^{n}\nu_{i}h_i$. By the orthogonality properties of $W_{\bullet}$ with respect to $\langle\cdot,\cdot\rangle$, we get $\langle h_i, h_{n-i}\rangle\neq 0$ and
\[y_i=\langle h_i,e\rangle=\sum_{j=i}^nc_{i,n-j}\nu_{n-j}\] for each $0\leq i\leq n$. Hence there are $d_{i,j}\in\mathbb{C}$ such that 
\[\nu_{n-i}=\sum_{j=0}^id_{i,j}y_j\] for each $0\leq i\leq n$.
The operators $\mathcal{N}_k$ are $\mathbb{C}$-linear, which implies $\mathcal{N}_k(\nu_{n-i})=0$ for all $0\leq i\leq k-1$.
As the $h_0,\dots,h_n$ form a flat basis of $M_L\otimes\PV_0$, we get \[\mathcal{N}_k(e)=\sum_{i=0}^{n-k}\mathcal{N}_k(\nu_i)h_i\] and hence $\mathcal{N}_k(e)\in W_{2(n-k)}$.
\end{proof}

As Proposition \ref{Vau} implies that $V^i\neq \{0\}$ for each $0\leq i\leq n$ and $V^i\cap V^j=\{0\}$ for $i\neq j$, we get 

\begin{cor}
Each $V^i$ is a $\mathbb{C}\llbracket z\rrbracket$-module of rank one and we have $M_L=\bigoplus_{i=0}^nV^i$ as $\mathbb{C}\llbracket z\rrbracket$-module.
\end{cor}

A further consequence of Proposition \ref{Vau} is the following symmetry among the structure series of $L$.

\begin{cor}\label{alphsym}
The structure series $\alpha_1,\dots,\alpha_n$ of $L$ fulfill
$\alpha_k=\alpha_{n+1-k}$ for all $k=1,\dots,n$.
\end{cor}

\begin{proof}
As in the proof of Proposition \ref{Vau}, we write $e=\sum_{i=0}^{n}\nu_{i}h_i$ and $\nu_{n-i}=\sum_{j=0}^id_{i,j}y_j$. Since $d_{i,i}\neq 0$ we have that 
\[\mathcal{N}_k(\nu_{n-k})=d_{k,k}\mathcal{N}_k(y_k)=d_{k,k}\alpha_k^{-1}.\] Moreover, the orthogonality properties of $E^{\bullet}$ and $W_{\bullet}$ with respect to $\langle\cdot,\cdot\rangle$ yield
\[\langle \mathcal{N}_i(e), \mathcal{N}_j(e)\rangle=\begin{cases}d_{i,i}d_{n-i,n-i}\left(\alpha_i\alpha_{n-i}\right)^{-1},&\ \textrm{ if }j=n-i\\ 0,&\ \textrm{else}\end{cases}.\] 
We conclude that 
\begin{align*}
0&=z\frac{d}{dz}\langle\alpha_{i-1}\mathcal{N}_{i-1}(e), \alpha_{n-i}\mathcal{N}_{n-i}(e)\rangle=\langle\mathcal{N}_i(e),\alpha_{n-i}\mathcal{N}_{n-i}(e)\rangle+\langle\alpha_{i-1}\mathcal{N}_{i-1}(e),\mathcal{N}_{n+1-i}(e)\rangle\\&=d_{i,i}\alpha_{i}^{-1}+d_{i+1,i+1}\alpha_{n+1-i}^{-1}.
\end{align*}
As by construction $\alpha_i(0)=\alpha_{n+1-i}(0)=1$, we get the result.
\end{proof}

By the preceding corollary, the local normal form of $L$ reads
\[\mathcal{N}(L)=\vartheta\alpha_1\vartheta\alpha_2\dots\vartheta\alpha_2\vartheta\alpha_1\vartheta.\]

We modify this normal form by a local change of coordinates. In the sequel, given an element $f\in z\mathbb{C}\llbracket z\rrbracket^{*}$, we write 
\[f^{*}\colon \mathbb{C}\llbracket z\rrbracket[\vartheta]\to \mathbb{C}\llbracket z\rrbracket[\vartheta],\ f^{*}\left(a\right):=a\circ f,\ f^{*}\vartheta:=\frac{f}{z\frac{df}{dz}}\vartheta\]
for its induced automorphism on $\mathbb{C}\llbracket z \rrbracket [\vartheta]$, where $\circ$ denotes the usual composition of functions. The compositorial inverse of this automorphism is denoted by $\left(f^{\vee}\right)^{*}$.

\begin{defin}
The uniquely determined solution  $q\in z\mathbb{Q}\llbracket z\rrbracket^{*}$  of the differential equation
\[z\frac{d}{dz}q=\alpha_1^{-1}q,\ \frac{d\omega}{dz}(0)=1\] is called the \textbf{special coordinate} or the \textbf{q-coordinate} of $L$ near $z=0$.
Moreover, we put 
\[\vartheta_q:=q^{*}\vartheta=\alpha_1\vartheta.\]
\end{defin}

In particular, we obtain
\[\mathcal{N}(L)=\vartheta^2_q\frac{\alpha_2}{\alpha_1}\vartheta_q\dots\vartheta_q\frac{\alpha_2}{\alpha_1}\vartheta^2_q.\]

\begin{defin}
We call \[\mathcal{N}(L)_q:=\left(q^{\vee}\right)^{*}(\mathcal{N}(L))=\vartheta^2\left(q^{\vee}\right)^{*}\left(\frac{\alpha_2}{\alpha_1}\right)\vartheta\dots\vartheta\left(q^{\vee}\right)^{*}\left(\frac{\alpha_2}{\alpha_1}\right)\vartheta^2\]
the \textbf{special local normal form} of $L$, where $q$ denotes the special coordinate of $L$.   
\end{defin}

We can use the special normal form to decide whether two operators can be transformed into each other. 

\begin{prop}
Consider two differential operators $L_1$ and $L_2$ which satisfy properties (M) and (P). Then there is an element $\psi\in z\mathbb{Q}\llbracket z\rrbracket^{*}$ with $\frac{d\psi}{dz}(0)=1$ such that $\psi^{*}L_1=L_2$ if and only if their special local normal forms coincide.
\end{prop}

\begin{proof}
Denote the $q$-coordinate of $L_i$ by $q_i$. If $\psi^{*}L_1=L_2$ one checks directly that $q_2=q_1\circ \psi$ holds. As we have $(f\circ g)^{*}=g^{*}\circ f^{*}$ and $(f\circ g)^{\vee}=g^{\vee}\circ f^{\vee}$ for each two elements $f,g\in z\mathbb{C}\llbracket z\rrbracket^{*}$ and taking normal forms commutes with $\psi^{*}$, we get
\begin{align*}
\mathcal{N}(L_2)_{q_2}&=\left(\left(q_1\circ\psi\right)^{\vee}\right)^{*}\circ\psi^{*}\mathcal{N}(L_1)=\left(\psi\circ\left(q_1\circ\psi\right)^{\vee}\right)^{*}\mathcal{N}(L_1)\\&=\left(\psi\circ\psi^{\vee}\circ q_1^{\vee}\right)^{*}\mathcal{N}(L_1)=\left(q_1^{\vee}\right)^{*}\mathcal{N}(L_1)=\mathcal{N}(L_1)_{q_1}.
\end{align*}
On the other hand, if $\mathcal{N}(L_1)_{q_1}=\mathcal{N}(L_2)_{q_2}$ we get
$L_2=\left(q_1^{\vee}\circ q_2\right)^{*}L_1$. As $q'_1(0)=q'_2(0)=1$ the same holds for $(q_1^{\vee}\circ q_2)'(0)$. 
\end{proof}

We just observed that the series appearing in the special local normal form of $L$ are fixed under local transformations and hence provide additional invariants.

\begin{defin}
For $i=1,\dots,n-2$ we call \[Y_i:=\left(\frac{\alpha_{i+1}}{\alpha_1}\right)^{-1}\circ q^{\vee}\] the $i$-th 
\textbf{Y-invariant} of $L$.
\end{defin}

Note that the symmetry of the structure series stated in Corollary \ref{alphsym} imply that $Y_i=Y_{n-i}$ for $i=1,\dots,n-2$.
Considering the pairing $\langle\cdot,\cdot\rangle$ in local coordinates, we additionally find

\begin{prop}\label{ProdY}
If $L$ is monic and $0\neq\alpha\in\mathbb{C}(z)$ such that $L\alpha=\alpha L^{\vee}$ holds, we have
\[\frac{z^n\alpha}{y_0^2\alpha_1^n}\circ q^{\vee}=\left\langle \frac{e}{y_0}, \vartheta^{n}_q\frac{e}{y_0}\right\rangle\circ q^{\vee}=c\prod_{i=1}^{n-2}Y_i,\]
for a constant $c\in\mathbb{C}^{*}$.
\end{prop}

\begin{proof}
We only take care of the equality sign on the right. Rewriting $\vartheta^{n}_q\frac{e}{y_0}=\sum_{i=0}^n\lambda_i\mathcal{N}_i(e)$ a direct computation shows that
$\lambda_n=\alpha_1^{n-1}/(\alpha_2\cdot\dots\cdot\alpha_{n-1})$. As $\langle e,\mathcal{N}_i(e)\rangle=0$ for $i=0,\dots,n-1$ and Corollary \ref{alphsym} gives $\langle e,\mathcal{N}_n(e)\rangle=cy_0\alpha_1^{-1}$ for an $c\in\mathbb{C}^{*}$,  
we get the result.
\end{proof}

From the discussion of the local normal form of $L$, we take the following properties into account for our characterization.

\begin{defin}
We say that $L$ satisfies \textbf{property (Q)} if its special coordinate is N-integral. Moreover, $L$ satisfies \textbf{property (S)} if all its structure series are N-integral. 
\end{defin}

Several reasons to claim property (Q) can be found in \cite{LLY}. To name a few, the special coordinate of a Calabi-Yau operator related to certain families of elliptic curves determines a modular form, see e.g. \cite{Doran}, while it gives rise to a Thompson series in case of a family of K3 surfaces with generic Picard-rank $19$, see \cite{Yui}.
Note, that by Proposition \ref{ProdY}, property (Q) implies property (S) if the degree of the operator is less or equal than five.

Properties (Q) and (S) imply the N-integrality of the Y-invariants of $L$.
These can presumably be used to count objects on related families in the following way:
There is a transformation $z\mapsto \lambda z$, suppose that the coefficients of each $Y_i$ are integers and minimal in the sense that there is no $N\in\mathbb{N}$ such that $N^m$ divides the $m$-th coefficient of all $Y_i$ for each $m\in\mathbb{N}$.
For each choice $\ell\in\mathbb{N}$, we can rewrite $Y_i$ as a \textit{Lambert series}
\[\mathscr{L}(Y_i,\ell)=1+\sum_{d=1}^{\infty}\frac{N_{i,d,\ell}d^{\ell}z}{1-z^d}.\] If $\ell$ is chosen arbitrarily we get that $N_{i,d,\ell}\notin\mathbb{Z}$. However, for Picard-Fuchs operators of Calabi-Yau threefolds it seems that $N_{1,d,3}\in\mathbb{Z}$ for all $d\in\mathbb{N}$. In the famous case of families of quintics in $\mathbb{P}^4$, it is proven that $5 N_{1,d,3}$ is precisely the number of rational curves of degree $d$ on the mirror of the family for $0\leq d\leq 4$, see e.g. \cite{Elling}. 
For most of the other families of Calabi-Yau threefolds and higher dimensional families such an analogy is not known. 
Hence, we do not take this property into account for our description, but will state some related observations later on.

\subsection{Differential operators of CY-type}

We have all ingredients to define differential operators of CY-type

\begin{defin}
An irreducible differential operator $L\in\mathbb{Q}[z,\vartheta]$ is called \textbf{of CY-type} if it fulfills each of the following properties:
\begin{enumerate}
 \item[(P)] $L$ is self-dual in the sense that there is an $0\neq\alpha\in\mathbb{Q}(z)$ such that $L\alpha=\alpha L^{\vee}$ holds.
 \item[(M)] All exponents of $L$ at $z=0$ are integers and equal.
 \item[(N)] $L$ has at $z=0$ an N-integral local solution $y=\sum_{m=0}^{\infty}A_mz^m\in\mathbb{Q}\llbracket z\rrbracket$, i.e. there is an $N\in\mathbb{N}$ such that $N^mA_m\in\mathbb{Z}$ for all $m\geq 0$.
 \item[(Q)] The special coordinate of $L$ at $z=0$ is N-integral, i.e. there are solutions $y_0, y_1$ of $L$ at $z=0$ such that
\[q=\exp\left(y_1/y_0\right)\in z\mathbb{Q}\llbracket z\rrbracket^{*}\] is N-integral.
 \item[(S)] All structure series of $L$ at $z=0$ are N-integral, i.e. there are N-integral power series $\alpha_1,\dots,\alpha_n\in\mathbb{Q}\llbracket z\rrbracket$ such that
\[L=\alpha_n\vartheta\alpha_{n-1}\vartheta\dots\alpha_1\vartheta\alpha_0\in\mathbb{Q}\llbracket z\rrbracket[\vartheta]\]
 \end{enumerate}
We call two operators $L,L'$ of CY-type to be \textbf{equivalent}, written $L\equiv L'$, if there are algebraic functions $f,g\in\mathbb{Q}(z)^{alg}$ such that $f$ is meromorphic at $z=0$ and $g(0)=0$ 
such that $g^{*}(L)\otimes f=L'$.
\end{defin}

As seen before, equivalent operators of Calabi-Yau type have the same special local normal forms. The converse statement fails for operators of order two and three but seems - according to our observations - to be true for operators of order four or higher.
For pullbacks which are holomorphic near the origin we find

\begin{lemma}\label{Trans}
Let $L\in\mathbb{Q}[z,\vartheta]$ be of CY-type and $\psi\in\mathbb{Q}(z)^{alg}$ which is near $z=0$ given by $\psi(z)=c_0z^h+\sum_{m=h+1}^{\infty}c_mz^m$ for an $h>0$. Then $\psi^{*}(L)$ also is of CY-type. 
\end{lemma}

\begin{proof}
We study the cases where $c_0=1$ first. As taking duals is compatible with pullbacks, $\psi^{*}(L)$ satisfies property (P). Since the composition of N-integral power series remains N-integral, it also satisfies property (N). The exponents of $\psi^{*}(L)$ at $z=0$ are those of $L$ at $z=0$ multiplied by $h$ which implies property (M).
As we can write $\psi(z)=z^he$ for an $e\in\mathbb{Q}\llbracket z\rrbracket^{*}$ with $e(0)=1$, we get that $\ln(\psi(z))=h\ln(z)+\ln(e)$.
Hence, for each frame $y_0,\dots, y_n$ of $L$, a frame of $\psi^{*}(L)$ at $z=0$ is given by
$y'_k=h^{-k}\left(y_k\circ \psi\right)$ for $0\leq k\leq n$. As the class of N-integral power series is closed under derivation, this implies property (S). The special coordinate $\tilde{q}$ of $\psi^{*}(L)$ is uniquely determined by 
\[z\frac{d}{dz}\tilde{q}=\frac{1}{h}z\frac{d}{dz}\left(\frac{y_1}{y_0}\right)\tilde{q} \textrm{ and }\frac{d}{dz}\tilde{q}(0)=1.\] Therefore it is related to the special coordinate $q$ of $L$ via $\tilde{q}^h=q$. The Theorem of Eisenstein  \cite{Eisen} 
assures that $\tilde{q}$ is N-integral, which yields property (Q).

As one can show by similar arguments that $(\lambda z)^{*} L$ is of CY-type for each $\lambda\in\mathbb{Q}$, we obtain the result.
\end{proof}

\section{Differential Galois groups of CY-type operators}
\subsection{General possibilities}
Having defined CY-type differential operators, we seek for a suitable classification of them on a purely algebraic level, which is yet to be done.
As a first step, we investigate possible differential Galois groups. 
By Corollary \ref{Gfirst}, the differential Galois group of such an operator of degree $n+1$ lies in $\Sp_{n+1}(\mathbb{C})$
if $n+1$ is even and in $\SO_{n+1}(\mathbb{C})$ if $n+1$ is odd. Moreover, property (M) assures that the differential Galois group contains a maximally unipotent element. As those elements are particularly regularly unipotent, we meet the assumptions of \cite[Proposition 2.2]{SaxlSeitz}
and get

\begin{prop}\label{DG}
The only possible differential Galois groups of a CY-type operator of order $n+1$ are 
\begin{enumerate}
\item The trivial group or $\mathbb{Z}/2\mathbb{Z}$ if $n=0$.
\item $\SL_2(\mathbb{C})$ acting in its $\Sym^{n}$-representation.
\item $\Sp_{n+1}(\mathbb{C})$ for $n+1$ even.
\item $\SO_{n+1}(\mathbb{C})$ for $n+1$ odd.
\item the exceptional group $G_2(\mathbb{C})$ for $n+1=7$.
\end{enumerate}
\end{prop}

\begin{proof}
For $n>0$ the results are a direct consequence of \cite[Proposition 2.2]{SaxlSeitz} and the representation theory of $G_2(\mathbb{C})$. Consider a CY-type operator $L$ of degree one. As $L$ is an irreducible G-operator, the theorems of D.V. and G.V. Chudnovky and N. Katz given in \cite[Chapter VIII and Chapter III.6]{Dwork} assure that all exponents of $L$ are rational. As by Corollary \ref{Exponentssymm} each non integral exponent of $L$ lies in $1/2+\mathbb{Z}$, there is an algebraic field extension $K\supset\mathbb{C}(z)$ of degree at most two which contains all solutions of $L$. Therefore, the differential Galois group of $L$ coincides with the Galois group of $K/\mathbb{C}(z)$ and we get the result.
\end{proof}

The result above implies that CY-type operators of degree one are exactly those, whose solution space is spanned by an algebraic function $y=\sqrt{P/Q}$ where $P,Q\in\mathbb{Q}[z]$ with $gcd(P,Q)=1$ such that $y$ is holomorphic near $z=0$.

As the category of differential modules is tannakian, the differential Galois group $G$ of a given differential operator $L$ can be computed via representation theory, i.e. by detecting submodules of associated differential modules. For instance, if $L$ is of CY-type and has degree four the only possibilities for its differential Galois group $G$ are $\SL_2(\mathbb{C})$ acting in its $\Sym^3$-representation and $\Sp_4(\mathbb{C})$ by Proposition \ref{DG}. Then by \cite{Hess}, we find that $G=\Sp_4(\mathbb{C})$ if and only if $\Sym^2(L)$ is an irreducible operator of order ten, which may be elaborate to check. In many cases, it is sufficient to look at the local monodromies. For instance, suppose that $L$ admits a singularity $s\in\mathbb{P}^1$ such that the Jordan form of the local monodromy $T_s$ admits a Jordan block of size two - as e.g. in the prominent case of a so called \textit{conifold point}. 
Then $T_s$ is not the symmetric cube of a $2\times 2$-matrix and $G$ has to be $\Sp_4(\mathbb{C})$.

In some other cases, we may apply the following criterion

\begin{lemma}\label{KritSL}
Consider a CY-type operator $L$ of degree $n+1$ with Y-invariants $Y_1,\dots,Y_{n-2}$ and differential Galois group $G$. Then
\begin{enumerate}
\item $G$ is a proper subgroup of $\Sp_{n+1}(\mathbb{C})$ if $n+1$ is even and $Y_1=1$.
 \item $G$ is a proper subgroup of $\SO_{n+1}(\mathbb{C})$ if $n+1>5$ is odd and $Y_2=1$.
 \end{enumerate}
\end{lemma}

\begin{proof}
Throughout the proof, we fix a flag $y_k=\sum_{j=0}^k\frac{1}{j!}\ln^j(z)f_{k-j}$ of $L$ at $z=0$ with the additional property that $f_1,\dots,f_n\in z\mathbb{Q}\llbracket z\rrbracket$. Moreover, we denote the action of $z\frac{d}{dz}$ by $\left(\cdot\right)'$. 
If $n+1$ is even and $Y_1=1$, we find that
\[\left(\frac{\left(\frac{y_2}{y_0}\right)'}{\left(\frac{y_1}{y_0}\right)'}\right)'=\left(\frac{y_1}{y_0}\right)'\] and therefore
\[\frac{\left(\frac{y_2}{y_0}\right)'}{\left(\frac{y_1}{y_0}\right)'}=\frac{y_1}{y_0}+c.\] By our assumption on the flag, we see that $c=0$ and
\[\left(\frac{y_2}{y_0}\right)'=\frac{y_1}{y_0}\left(\frac{y_1}{y_0}\right)'=\frac{1}{2}\left(\frac{y_1^2}{y_0^2}\right)'\]
Again by the choice of the flag, this implies
\[y_0y_2-\frac{1}{2}y_1^2=0.\] Therefore, $\Sym^2(M_L)$ is reducible by \cite[Corollary 2.23]{Put}. As $\Sym^2(\Sp_{n+1}(\mathbb{C}))$ is irreducible, we get the first statement.
For the second statement, we have
\[\left(\frac{\left(\frac{\left(\frac{y_3}{y_0}\right)'}{\left(\frac{y_1}{y_0}\right)'}\right)'}{\left(\frac{\left(\frac{y_2}{y_0}\right)'}{\left(\frac{y_1}{y_0}\right)'}\right)'}\right)'=\left(\frac{y_1}{y_0}\right)'.\]
This implies
\[\left(\frac{\left(\frac{y_3}{y_0}\right)'}{\left(\frac{y_1}{y_0}\right)'}\right)'=\frac{y_1}{y_0}\left(\frac{\left(\frac{y_2}{y_0}\right)'}{\left(\frac{y_1}{y_0}\right)'}\right)'=\left(\frac{y_2}{y_0}-\frac{y_1\left(\frac{y_2}{y_0}\right)'}{y_0\left(\frac{y_1}{y_0}\right)'}\right)'\]
By our choice of a frame, we hence get
\[\left(\frac{y_3}{y_0}\right)'=\left(\frac{y_1}{y_0}\right)'\left(\frac{y_2}{y_0}\right)-\frac{y_1}{y_0}\left(\frac{y_2}{y_0}\right)'=-\left(\frac{y_1}{y_0}\right)^2\left(\frac{y_2}{y_1}\right)'\]
and thus that
\[\Wr(y_3,y_0)=y_0^2\left(\frac{y_3}{y_0}\right)'=-y_1^2\left(\frac{y_2}{y_1}\right)'=-\Wr(y_2,y_1).\] Hence $\bigwedge^2(M_L)$ is reducible by \cite[Corollary 2.28]{Put}, while $\bigwedge^2(\SO_{n+1}(\mathbb{C}))$ is not, which gives the second result.
\end{proof}

Combined with Proposition \ref{DG}, we immediately get
\newpage
\begin{cor}
Consider a CY-type operator $L$ of degree $n+1\neq 7$ with Y-invariants $Y_1,\dots,Y_{n-1}$ and differential Galois group $G$. Then $G\cong \SL_2(\mathbb{C})$ if
\begin{enumerate}
 \item $n+1$ is even and $Y_1=1$.
\item $n+1$ is odd and $Y_2=1$.
\end{enumerate}
\end{cor}

In the next sections, we investigate some examples of differential operators with differential Galois group $\SL_2(\mathbb{C})$ and $G_2(\mathbb{C})$.

\subsection{The $\SL_2(\mathbb{C})$-case}

By Proposition \ref{DG}, the differential Galois group of a second order CY-type operator is $\SL_2(\mathbb{C})$. The natural candidates are Picard-Fuchs operators for rational relatively minimal families of elliptic curves with section. Indeed, we find

\begin{prop}
The Picard-Fuchs operator of each rational relatively minimal rational family of elliptic curves with section which has a fiber of type $I_b$, $b>0$ at $z=0$ is of CY-type. 
\end{prop}

\begin{proof}
According to the work \cite{Klein} of R. Fricke and F. Klein, each such equation can be written as rational pullback of the differential operator \[E:=\partial^2-\frac{1}{z}\partial+\frac{31/144z-1/36}{z^2(z-1)^2}\]
followed by multiplying the solutions with a function which is algebraic of degree two over $\mathbb{Q}(z)$, meromorphic at $z=0$ and whose exponents lie in $\frac{1}{2}\mathbb{Z}$.
We have
\[\tilde{E}=\left(\frac{1}{z}\right)^{*}(E)=144\,{\vartheta}^{2}+z \left( 31-288\,{\vartheta}^{2} \right) +4\,{z}^{2} \left( 6\,\vartheta-1 \right)  \left( 6\,\vartheta+1
 \right).\] By Lemma \ref{Trans} it suffices to check that $\tilde{E}$ is of CY-type, as $z=0$ is the only singularity of $\tilde{E}$ with maximally unipotent monodromy. 
While property (M) is obvious, we get property (P) by direct computation. Moreover, a holomorphic solution of $\tilde{E}$ at $z=0$ is given by the N-integral power series $_2F_1(1/12,5/12;1\mid z)(1-z)^{-1/4}$
which gives property (N). To check property (Q), one observes that the special coordinate $q$ of $\tilde{E}$ is related to the famous modular form of Jacobi via
\[(1728z)^{*}\left(\frac{1}{q^{\vee}}\right)=z^{-1}+744z+196844z^2+\dots,\] see e.g. \cite{Doran}. Finally, as $\tilde{E}$ is of order two it obviously satisfies property (S) and thus is of CY-type.
\end{proof}

The results stated in \cite[Proposition 4.26]{Put}
yield a method to produce differential operators of CY-type of arbitrary degree whose differential Galois-group is $\SL_2(\mathbb{C})$.

\begin{prop}\label{SymCY}
Consider a CY-type operator $L$ of degree two. Then $\Sym^n(L)$ is a CY-type operator of degree $n+1$. 
\end{prop}

\begin{proof}
As a direct consequence of the first statement given in \cite[Proposition 4.26]{Put} and the representation theory of $\SL_2(\mathbb{C})$, each of the operators $\Sym^n(L)$ is irreducible and has degree $n+1$.
Denote the associated marked differential module by $\left(M_{\Sym^n(L)}, e\cdots e\right)$.
By property (P) we have a non-degenerate anti-symmetric pairing $\langle\cdot,\cdot\rangle$ on $M_L$. Then the induced pairing $(\cdot,\cdot)$ on $M_{\Sym^n(L)}$ is non-degenerate, $(-1)^n$-symmetric and we have $\left(e\cdots e,\partial^k(e\cdots e)\right)=0$ for $k=0,\dots,n$. Hence $\Sym^n(L)$ fulfills property (P) by Proposition \ref{CY2anders}.  
Consider a flag $y_0,y_1$ of $L$ at $z=0$. Then \cite[Proposition 4.26.3]{Put} assures that $w_k=\frac{1}{k!}y^{n-k}_0y^k_1$ is a flag of $\Sym^n(L)$ at $z=0$. As in particular $w_0=y^n_0$, the operator $\Sym^n(L)$ fulfills property (N). To show property (M), assume without loss of generality that the exponents of $L$ at $z=0$ are all equal to zero. 
The recursion formula for $\Sym^n(L)$ given in \cite[Proposition 4.26.2]{Put} inductively yields that the sum of the exponents of $\Sym^n(L)$ at $z=0$ equals zero. Moreover, the exponent of $w_0$ is zero and this function spans the space of holomorphic solutions of $\Sym^n(L)$ at $z=0$, which implies that the biggest exponent of $\Sym^n(L)$ at $z=0$ is zero as well.
As $\Sym^n(L)$ satisfies property (P), the symmetry of the exponents stated in Corollary \ref{Exponentssymm} forces all of them to be zero, which gives property (M). 
The special coordinate $q$ of $\Sym^n(L)$
fulfills the differential equation
\[z\frac{dq}{dz}=z\frac{d}{dz}\left(\frac{w_1}{w_0}\right)q=\left(\frac{y_1}{y_0}\right)\] and hence coincides with the special coordinate of $L$ which gives property (Q). Finally, one checks inductively that all structure series of $\Sym^n(L)$ are equal, i.e.
\[\alpha_1=\dots=\alpha_n=z\frac{dq}{dz}.\] As the N-integrality of $q$ implies the N-integrality of $\alpha_1$, we obtain property (S).
\end{proof}

We can actually check whether a CY-type operator is a symmetric power via its Y-invariants.

\begin{prop}\label{YSL}
Consider a CY-type operator $L$ of degree $n+1$. There is an algebraic function $g\in\mathbb{Q}(z)^{alg}$ such that $L\otimes g$ can be written as $n$-th symmetric power of a CY-type operator of degree two if and only if all Y-invariants of $L$ are equal to one. 
\end{prop}

\begin{proof}
If $P$ is a CY-type operator of order two, we have seen in the proof of Proposition \ref{SymCY} that all structure series of $\Sym^{n+1}(P)$ are equal and hence all of its Y-invariants are equal to one. As $\Sym^{n+1}(P)$ and $\Sym^{n+1}(P)\otimes g$  have the same structure series this gives one part of the equivalence. Conversely, suppose that all Y-invariants of $L$ are equal to one. By an appropriate choice of $g$, we may assume without loss of generality that the exponents of $L$ at $z=0$ are all equal to zero and that there is an $0\neq \alpha\in\mathbb{Q}(z)$ which fulfills $\alpha'=-2a_n/(n+1)\alpha$ and admits an $n$-th root $\beta\in\mathbb{Q}(z)$. Choose a flag $y_k=\sum_{j=0}^k\frac{1}{j!}\ln^j(z)f_{k-j}$ of $L$ at $z=0$ with the additional property that $f_0\in\mathbb{Q}\llbracket z\rrbracket^{*}$, $f_0(0)=1$ and $f_1,\dots,f_k\in z\mathbb{Q}\llbracket z\rrbracket$. As observed in the proof of Proposition \ref{KritSL}, we have \[\frac{y_1}{y_0}=2\frac{y_2}{y_1}.\] In a similar manner, an iteration yields
\[\frac{y_1}{y_0}=2\frac{y_2}{y_1}=6\frac{y_3}{y_2}=\dots=n!\frac{y_{n}}{y_{n-1}}.\] Setting $w_0:=(y_0)^{\frac{1}{n}}\in\mathbb{Q}\llbracket z\rrbracket$ and $w_1:=w_0\frac{y_1}{y_0}\in\mathbb{Q}\llbracket z\rrbracket[\ln(z)]$ we hence find that
$\prod_{j=1}^kj!y_k=w^{n-k}_0w^k_1$ holds. We put \[P:=\frac{\Wr(\cdot,w_0(x),w_1(x))}{\Wr(w_0(x),w_1(x))}=\partial^2+b_1\partial+b_2\in\mathbb{Q}\llbracket z\rrbracket[\partial]\] and see that the localization of $L$ at $z=0$ coincides with $\Sym^{n}(P)$ by \cite[Proposition 4.26.3]{Put}. Writing $L=\partial^{n+1}+\sum_{i=0}^{n}a_{i}\partial^i$, the iteration process given in \cite[Proposition 4.26.2]{Put} inductively yields that $a_{n}=n(n+1)b_1/2$ while $b_2$ is a $\mathbb{Q}$-linear combination of $a_{n}', a_{n}^2$ and $a_{n-1}$. Therefore, we actually get that $P\in \mathbb{Q}(z)[\partial]$ and $\Sym^{n}(P)=L$.

It remains to check that $P$ is of CY-type. We first note that the irreducibility of $L$ implies the irreducibility of $P$. Moreover, as $P$ is of order two, it automatically fulfills properties (Q) and (S). Property (N) follows from the N-integrality of $w_0$. By our assumptions, $\beta$ is rational and fulfills the differential equation \[\beta'=-2a_n/(n(n+1))\beta=-b_1\beta.\] This implies property (P). The sum of the exponents of $L$ at $z=0$ is zero. Solving the differential equation $\alpha'=-2a_n\alpha/(n+1)$ locally hence yields that $\alpha$ has a pole of order $n$ at $z=0$. Therefore, $\beta$ has a pole of order one at $z=0$. By the differential equation stated above, the residue of $b_1$ at $z=0$ is equal to one. Then \[\Ind_0(P)=T(T-1)+T+\Res_{z=0}(zb_2)\] and the sum of the exponents of $P$ at $z=0$ is zero. As the exponent of $w_0$ is zero as well, this yields property (M).      
\end{proof}

As a direct consequence we find a special case of a classical result due to G. Fano, see \cite{Fano}.

\begin{cor}
Up to tensor products with algebraic functions, each CY-type operator of degree three is the symmetric square of a CY-type operator of degree two. 
\end{cor}

It is not clear whether each CY-type operator whose differential Galois group is $\SL_2(\mathbb{C})$ can be written as a symmetric power of a CY-type operator of order two although all known examples are of this type. As the differential Galois group of an operator is an invariant of the associated differential module, this is closely related to the question which elements of a given differential module give rise to a differential operator of CY-type.

\subsection{The $G_2(\mathbb{C})$-case}

Recently, first examples of CY-type operators whose differential Galois group is $G_2(\mathbb{C})$ were found. They are contained in \cite{DRGzwei} of Dettweiler and Reiter and read
\begin{align*}R_1&={\vartheta}^{7}-128\,z \left( 48\,{\vartheta}^{4}+96\,{\vartheta}^{3}+
124\,{\vartheta}^{2}+76\,\vartheta+21 \right)  \left( 2\,\vartheta+1
 \right) ^{3}\\&\quad+4194304\,{z}^{2} \left( \vartheta+1 \right)  \left( 12\,
{\vartheta}^{2}+24\,\vartheta+23 \right)  \left( 2\,\vartheta+1
 \right) ^{2} \left( 2\,\vartheta+3 \right) ^{2}\\&\quad-34359738368\,{z}^{3}
 \left( 2\,\vartheta+5 \right) ^{2} \left( 2\,\vartheta+1 \right) ^{2}
 \left( 2\,\vartheta+3 \right) ^{3}\\
R_2&= {\vartheta}^{7}-128\,z \left( 8\,{\vartheta}^{4}+16\,{\vartheta}^{3}+
20\,{\vartheta}^{2}+12\,\vartheta+3 \right)  \left( 2\,\vartheta+1
 \right) ^{3}\\&\quad+1048576\,{z}^{2} \left( 2\,\vartheta+1 \right) ^{2}
 \left( 2\,\vartheta+3 \right) ^{2} \left( \vartheta+1 \right) ^{3}\\
R_3&={\vartheta}^{7}-27\,z \left( 3\,\vartheta+2 \right)  \left( 2\,
\vartheta+1 \right)  \left( 3\,\vartheta+1 \right)  \left( 81\,{
\vartheta}^{4}+162\,{\vartheta}^{3}+198\,{\vartheta}^{2}+117\,
\vartheta+28 \right)\\&\quad +531441\,{z}^{2} \left( 3\,\vartheta+5 \right) 
 \left( 3\,\vartheta+1 \right)  \left( \vartheta+1 \right)  \left( 3\,
\vartheta+4 \right) ^{2} \left( 3\,\vartheta+2 \right) ^{2}
\\
R_4&={\vartheta}^{7}-128\,z \left( 4\,\vartheta+1 \right)  \left( 2\,
\vartheta+1 \right)  \left( 4\,\vartheta+3 \right)  \left( 128\,{
\vartheta}^{4}+256\,{\vartheta}^{3}+304\,{\vartheta}^{2}+176\,
\vartheta+39 \right)\\&\quad +67108864\,{z}^{2} \left( 4\,\vartheta+7 \right) 
 \left( 4\,\vartheta+3 \right)  \left( 2\,\vartheta+3 \right)  \left( 
2\,\vartheta+1 \right)  \left( 4\,\vartheta+5 \right)  \left( 4\,
\vartheta+1 \right)  \left( \vartheta+1 \right)\\ 
R_5&={\vartheta}^{7}-3456\,z \left( 6\,\vartheta+5 \right)  \left( 2\,
\vartheta+1 \right)  \left( 6\,\vartheta+1 \right)  \left( 648\,{
\vartheta}^{4}+1296\,{\vartheta}^{3}+1476\,{\vartheta}^{2}+828\,
\vartheta+155 \right)\\&\quad +557256278016\,{z}^{2} \left( 6\,\vartheta+11
 \right)  \left( 6\,\vartheta+5 \right)  \left( 3\,\vartheta+5
 \right)  \left( 3\,\vartheta+1 \right)  \left( 6\,\vartheta+7
 \right)  \left( 6\,\vartheta+1 \right)  \left( \vartheta+1 \right) 
 \end{align*}
The first terms of the special coordinates of these operators read
\begin{align*}
R_1&:\ z+7040\,{z}^{2}+67555904\,{z}^{3}+747082784768\,{z}^{4}+
8968272297124128\,{z}^{5}+\dots\\
R_2&:\ z+1152\,{z}^{2}+2150976\,{z}^{3}+4983447552\,{z}^{4}+13054714896672{z}^{5}+\dots\\
R_3&:\ z+5562\,{z}^{2}+49552317\,{z}^{3}+547802062578\,{z}^{4}+
6855142017357054\,{z}^{5}+\dots\\
R_4&:\ z+72576\,{z}^{2}+8462979648\,{z}^{3}+1230038144557056\,{z}^{4}+
203018472128017391904\,{z}^{5}+\dots\\
R_5&:\ z+20200320\,{z}^{2}+689499895026240\,{z}^{3}+29916247864887732510720
\,{z}^{4}\\&\quad\, +1488739080271271648779215102240\,{z}^{5}+\dots
\end{align*}

Computing the Y-invariants $Y_1$ and $Y_2$ in all of the examples reveals that $Y_2$ is constant while $Y_{1}$ is not.
Moreover, each of the series $Y_{1}$ seems to admit an integral Lambert-series expansion $\mathscr{L}(Y_{1},4)$. The first five numbers $N_{1,d,4}$ of these expansions read 
\begin{align*}
R_1&:\ 768,-136800,35597568,-5313408000,-6059212935936\\
R_2&:\ 256,45504,20254464,14135932800,12870108663552\\
R_3&:\ 1485,9853515/8,2555194005,8549298943740,37455896889425700\\
R_4&:\ 29440,277414560,7671739956480,346114703998149120,20536396999367861894400\\
R_5&:\ 17342208,42976872163296,380850322188446486784,5581133974953140362085043072,\\&\quad\, 108045504354230644224717527051669760
\end{align*}

It would be interesting to know if any of these series have a further arithmetical or geometrical meaning.

We make some final remarks on relations between the coefficients of these examples, all of which can be checked by straightforward computations.
If we consider a CY-type operator $L$ of order seven, there is a rational function $g$ such that \[R:=L\otimes (\partial-g'/g)=\partial^7+\sum_{i=0}^5a_i\partial^i\in\mathbb{C}(z)[\partial].\]
Property (P) - which holds for $R$ as well - implies the relations
\begin{align*}
a_4&=-\frac{5}{2}a'_5\\
a_2&= -\frac{5}{2}a'''_5+\frac{3}{2}a'_3\\
a_0&=\frac{1}{2}a'_1-\frac{1}{4}a''_3+\frac{1}{2}a^{(5)}_5.
\end{align*}
If additionally $Y_2=1$, expanding the local normal form yields the relation
\[a_3=3a''_5+\frac{1}{4}a^2_5\] in a neighborhood of $z=0$ and hence in $\mathbb{C}(z)$ as well.
However, we do not think that the relation on $a_3$ stated above implies $Y_2=1$.
Finally, by Proposition \ref{YSL}, the relation on $a_3$ stated above together with the additional relation
\begin{align*}
a_1=\frac{5}{7}a_5^{(4)}+\frac{22}{49}a''_5a_5+\frac{295}{784}(a'_5)^2+\frac{9}{686}a_5^3 
\end{align*}
is equivalent $Y_1=Y_2=1$.

\providecommand{\bysame}{\leavevmode\hbox to3em{\hrulefill}\thinspace}
\providecommand{\MR}{\relax\ifhmode\unskip\space\fi MR }
\providecommand{\MRhref}[2]{%
  \href{http://www.ams.org/mathscinet-getitem?mr=#1}{#2}
}
\providecommand{\href}[2]{#2}

\end{document}